\documentclass[12pt]{amsart}
\usepackage{amsmath,amsthm,amsfonts,amssymb}
\usepackage{graphics,color}
\usepackage{hyperref}

\newtheorem{theorem}{Theorem}
\newtheorem{proposition}{Proposition}
\newtheorem{remark}{Remark}
\newtheorem{lemma}{Lemma}

\numberwithin{equation}{section}

\begin{document}

\title[Uniqueness and blowup properties for SDEs]{On uniqueness and blowup properties for a class of second order SDEs %related to the stochastic wave equation
}
\author[Gomez, Lee, Mueller, Neuman, and Salins]{Alejandro Gomez \and Jong Jun Lee \and Carl Mueller \and Eyal Neuman \and Michael Salins}

\address{Alejandro Gomez}
\email{gomezalejandroh@gmail.com}

\address{Jong Jun Lee: Dept. of Mathematics
\\University of Rochester
\\Rochester, NY  14627}
\email{jlee263@ur.rochester.edu}

\address{Carl Mueller: Dept. of Mathematics
\\University of Rochester
\\Rochester, NY  14627 }
\urladdr{http://www.math.rochester.edu/people/faculty/cmlr}

\address{Eyal Neuman: Dept. of Mathematics
\\Imperial College London
\\ London, UK  SW7 2AZ}
\urladdr{http://eyaln13.wixsite.com/eyal-neuman}

\address{Michael Salins: Dept. of Mathematics and Statistics
\\Boston University
\\ Boston, MA 02215}
\urladdr{http://math.bu.edu/people/msalins/}

\begin{abstract}
As the first step for approaching the uniqueness and blowup properties of the solutions of the stochastic wave equations with multiplicative noise, we analyze the conditions for the uniqueness and blowup properties of the solution $(X_t,Y_t)$ of the equations $dX_t= Y_tdt$, $dY_t = |X_t|^\alpha dB_t$, $(X_0,Y_0)=(x_0,y_0)$. In particular, we prove that solutions are nonunique if $0<\alpha<1$ and $(x_0,y_0)=(0,0)$ and unique if $1/2<\alpha<1$ and $(x_0,y_0)\neq(0,0)$. We also show that blowup in finite time holds if $\alpha>1$ and $(x_0,y_0)\neq(0,0)$.

\end{abstract}
\maketitle

\section{Introduction and Main Results}
\label{section:introduction}

The basic uniqueness theory for ordinary differential equations (ODE) has
been well understood for a long time.  If $F(u)$ is a Lipschitz continuous function,
then
\begin{equation*}
\dot{u}(t)=F(u), \qquad u(0)=u_0
\end{equation*}
has a unique solution valid for all time $t\geq0$.  Furthermore, the
Lipschitz condition on the coefficients cannot be weakened to H\"older
continuity with index less than 1.

The situation for stochastic differential equations (SDE) is very different.
The classical Yamada-Watanabe theory of strong uniqueness \cite{yw71} states
that if $f(x)$ is a locally H\"older continuous function of index $1/2$ with
at most linear growth, then
\begin{equation*}
dX=f(X)dW, \qquad X_0=x_0
\end{equation*}
has a unique strong solution valid for all time $t\geq0$.  The H\"older
continuity condition cannot be weakened to indices below $1/2$.  Besides the
H\"older $1/2$ condition, another notable difference from the ODE case is
that the Yamada-Watanabe uniqueness result for SDE is essentially a
one-dimensional result.  That is, much less is known for vector-valued SDE,
whereas the above statement for ODE is still true in the case of
vector-valued solutions.

The basic conditions for uniqueness of partial differential equations (PDE)
are the same as for ODE: coefficients must be Lipschitz continuous.  But the
corresponding results for stochastic partial differential equations (SPDE)
have only appeared recently.  These results are restricted to the stochastic
heat equation,
\begin{align}\label{she}
\partial_t u&=\Delta u+f(u)\dot{W}  \\
\notag u(0,x) &= u_0(x).
\end{align}
Here $x\in\mathbf{R}$, $\dot{W}=\dot{W}(t,x)$ is two-parameter white noise,
and $f$ is H\"older continuous with index $\gamma$.  In this case, strong
uniqueness holds for $\gamma>3/4$ \cite{mp10}, but fails for $\gamma<3/4$
\cite{mmp14}.  One can also replace white noise by colored noise, which may
allow $x$ to take values in $\mathbf{R}^d$ for $d>1$, and may change the
critical value of $\gamma$.

The counterexample in \cite{mmp14} which proved nonuniqueness for
$\gamma<3/4$ involved the equation
\begin{align*}
\partial_t u&=\Delta u+|u|^\gamma\dot{W}  \\
u(0,x) &= 0.
\end{align*}
In fact, the case of $\gamma=1/2$ is the well-studied case of super-Brownian
motion, also called the Dawson-Watanabe process, see \cite{daw93},
\cite{per02}.

Other types of SPDE than the stochastic heat equations are still unexplored with regard
to uniqueness, except for the standard fact that uniqueness holds with
Lipschitz coefficients.  For example, there is no information about the
critical H\"older continuity of $f(u)$ for uniqueness of the stochastic wave
equation:
\begin{align}\label{swe}
\partial_t^2 u&=\Delta u+f(u)\dot{W}  \\
\notag u(0,x) &= u_0(x), \qquad \partial_tu(0,x) = u_1(x).
\end{align}
Here again $x\in\mathbf{R}$ and $\dot{W}=\dot{W}(t,x)$ is two-parameter
white noise.

In order to shed light on uniqueness for the stochastic wave equation, we
propose studying the corresponding SDE $\ddot{X}=f(X)\dot{B}$.  By making this
equation into a system of first order equations, we arrive at the equations
\begin{align}
\label{wave-sde}
dX &= Ydt  \nonumber \\
dY &= |X|^\alpha dB  \\
(X_0, Y_0) &= (x_0, y_0).  \nonumber
\end{align}
Here $B=B_t$ is a standard Brownian motion, and we use the subscripts $X_t$ or
$Y_t$ to indicate dependence on time, rather than $X(t)$ or $Y(t)$.  Here we
focus on the coefficient $f(x)=|x|^\alpha$ because this function had special
importance in the stochastic heat equation, and it is a prototype of a
function which is H\"older continuous of order $\alpha$.

Now we are ready to present our main results. In our first theorem, we show that when $\alpha>1/2$ and the initial condition is nonzero, strong uniqueness holds for the solutions of (\ref{wave-sde}) up to the hitting time of the origin.
\begin{theorem}
\label{t1}
If $\alpha>1/2$ and $(x_0,y_0)\ne(0,0)$, then (\ref{wave-sde}) has a
unique solution in the strong sense, up to the time $\tau$ at which the solution $(X_t,Y_t)$ first takes the value $(0,0)$.
\end{theorem}
In the next theorem, we prove that when $\alpha>1/2$, the unique strong solution of (\ref{wave-sde}) from Theorem \ref{t1} never reaches the origin.
\begin{theorem}
\label{t2}
If $\alpha>1/2$ and $(x_0,y_0)\ne(0,0)$, then the unique strong solution
$(X_t,Y_t)$ to (\ref{wave-sde}) never reaches the origin.  That is, the time
$\tau$ defined in Theorem 1 is infinite almost surely.
\end{theorem}
In our next result, we prove the nonuniqueness for the solutions of (\ref{wave-sde}) initiated at the origin.
\begin{theorem}
\label{t3}
If $0<\alpha<1$ and $(x_0,y_0)=(0,0)$, then both strong and weak uniqueness
fail for (\ref{wave-sde}).
\end{theorem}
A few remarks are in order.  \\
\textbf{Remarks:}
\begin{enumerate}
\item The proof of Theorem \ref{t1} builds on the Yamada-Watanabe
argument, as do the vast majority of strong uniqueness proofs for SDE, which go beyond the case of Lipschitz coefficients.
\item The proofs of Theorems \ref{t2} and \ref{t3} rely on a time-change
argument.  The proofs of Theorems 2 and 3 rely on a time-change argument, 
and the idea is inspired by Girsanov’s nonuniqueness example for SDE (see 
e.g. Example 1.22 in Chapter 1.3 of \cite{ce05}).
\item Note that the coefficient $|x|^\alpha$ is Lipschitz continuous except in a neighborhood of $x=0$.
\end{enumerate}

Now we turn our attention to the question of blowup in finite time. In the case of stochastic heat equation \eqref{she}, the critical  H\"older continuity index $\gamma$ of $f$ is $3/2$. If $\gamma>3/2$, then the solution blows up in finite time with positive probability (see \cite{ms93},\cite{mue00}). For $\gamma<3/2$, the solution does not blow up almost surely \cite{mue91l}. It is still unknown what happens when $\gamma=3/2$.

The blowup property of the stochastic wave equation appears to be more
difficult to analyze. It is still not known what conditions on $f$ give
finite time blowup of the solution of \eqref{swe} (see \cite{mr14}).
Sufficient conditions for the divergence of the expected $L^{2}$ norm of the
solutions in finite time were derived by Chow in \cite{chow09}. This result
however is insufficient to establish the almost sure blowup of the
solutions to (\ref{swe}). 
 We study the solution of \eqref{wave-sde} as the first step for approaching the stochastic wave equation.

The finite time blowup of the solutions of the first order stochastic differential equations can be checked by the Feller test for explosions (for example, see \cite{im74}); however, there is not a simple way to check  in the case of higher order equations. It is well-known that the solution of \eqref{wave-sde} doesn't blow up if the coefficients have at most linear growth (that is $\alpha \le 1$). In the next theorem, we prove that when $\alpha >1$, the solution of (\ref{wave-sde}) blows up in finite time with probability one. Before stating the theorem, we define some stopping times.

For any solution $(X_t,Y_t)$ of (\ref{wave-sde}), let
\begin{equation*}
\sigma^{X}_{L}:= \inf\{t> 0: \ |X_{t}| \geq L\}
\end{equation*}
 and
\begin{equation*}
\sigma^{X}: =\lim_{L\rightarrow \infty }\sigma^{X}_{L}.
\end{equation*}
$\sigma^Y$ can be defined analogously. Then, the following theorem holds.

\begin{theorem} \label{t4}
Assume that $\alpha >1$ and $(x_0,y_0)\neq (0,0)$. Then, the solution of (\ref{wave-sde}) satisfies
\begin{equation*}
\sigma^{X}=\sigma^{Y} <\infty
\end{equation*}
almost surely. Moreover, $|(X_t,Y_t)|_{\ell^\infty} \rightarrow \infty$ as $t \rightarrow \sigma^X$, where $|(x,y)|_{\ell^\infty}=|x|\vee|y|$ is the $\ell^\infty$ norm.
\end{theorem}

We now give some remarks.

\noindent\textbf{Remarks:}
\begin{enumerate}
\item The result of Theorem \ref{t4} is derived by showing that the blowup property of the solutions of (\ref{wave-sde}) follows from the transience property of a simplified time changed system. By proving that the inverse time change transforms infinite time to a finite time, we establish the finite time  blowup property.
\item From the proof of Theorem \ref{t4} it follows that  $|X_t|$ and $|Y_t|$ will fluctuate up and down as $t \rightarrow \sigma^X$ and won't converge to any number in $\mathbf{R} \cup \{\infty\}$. However, due to the correlation between them, $|X_t|\vee|Y_t| \rightarrow \infty$ as $t \rightarrow \sigma^X$ (see Remark \ref{remark-J-B} in Section \ref{sec-t4}).
\end{enumerate}
\paragraph{\textbf{Structure of the paper.}} The rest of this paper is dedicated to the proofs of Theorems \ref{t1}--\ref{t4}. In Section \ref{sec-t1}, we prove Theorem \ref{t1}. Section \ref{sec-t3} is devoted to the proof of Theorem \ref{t3}. In Sections  \ref{sec-t2} and \ref{sec-t4}, we prove Theorems \ref{t2} and \ref{t4} respectively.

\section{Proof of Theorem \ref{t1}} \label{sec-t1}

Let $(X^i_t,Y^i_t):i=1,2$ be two solutions to (\ref{wave-sde}) starting from $(x_0,y_0) \neq (0,0)$ and $\tau$ be the first time $t$ that either $(X^1_t,Y^1_t)$ or
$(X^2_t,Y^2_t)$ hits the origin.  Let $\tau_n$ for a natural number $n$ be the first time $t$ at which either
\begin{equation*}
|(X^1_t,Y^1_t)|_{\ell^\infty}\wedge|(X^2_t,Y^2_t)|_{\ell^\infty}\leq2^{-n}\end{equation*}
or
\begin{equation*}
|(X^1_t,Y^1_t)|_{\ell^\infty}\vee|(X^2_t,Y^2_t)|_{\ell^\infty}\geq2^{n}.
\end{equation*}
%where $|(x,y)|_{\ell^\infty}=|x|\vee|y|$ is the $\ell^\infty$ norm. \medskip \\
Since the coefficients of (\ref{wave-sde}) have at most linear growth, we have $|(X^1_t,Y^1_t)|_{\ell^\infty}\vee|(X^2_t,Y^2_t)|_{\ell^\infty} <\infty$ almost surely. As a result,
\begin{equation} \label{tau}
\lim_{n\to\infty}\tau_n=\tau.
\end{equation}
Note that it is possible that $\tau=\infty$.

We will show uniqueness up to time $\tau_n$ for each fixed $n$. Let $(X^{i,n}_t,Y^{i,n}_t)$ be the processes after stopping the noise at time $\tau_n$, that is
\begin{align}
\label{wave-sde-truncated}
dX^{i,n}_t &= Y^{i,n}_tdt  \nonumber\\
dY^{i,n}_t &= |X^{i,n}_t|^\alpha \mathbf{1}_{[0,\tau_n]}(t)dB_t  \\
X^{i,n}_0 &= x_0, \qquad  Y^{i,n}_0 = y_0.  \nonumber
\end{align}
So, $Y_t^{i,n}$ is constant for $t\geq\tau_n$.  We claim that for each $i=1$,2, there is at most one time $t>\tau_n$ at which $X^{i,n}_t = 0.$
Indeed, if $Y^{i,n}_{\tau_n}=0$, then $X^{i,n}_t$ is constant for
$t\geq\tau_n$ and this constant cannot be 0 because
$|(X^{i,n}_{\tau_n},Y^{i,n}_{\tau_n})|_{\ell^\infty}\ne0$.  In this case, there is no time
$t\geq\tau_n$ at which $X^{i,n}_t=0$.  But if $Y^{i,n}_t$ is a nonzero
constant for $t\geq\tau_n$, then $X^{i,n}_t$ is a nonconstant affine function
of $t$ for $t\geq\tau_n$, and so equals 0 at most once for $t\geq\tau_n$.

We will also define stopping times $\sigma^i_1<\sigma^i_2<\cdots$ as the successive times $t$ at which $X^{i,n}_t=0$.  We claim that with
probability 1, there are only finitely many such times.  The preceding
argument shows that for $i$ fixed, there is at most one value of $k$ for
which $\sigma^i_k>\tau_n$.  For $t<\tau_n$, since
$|(X^{i,n}_t,Y^{i,n}_t)|_{\ell^\infty}>2^{-n}$, we see that once $X^{i,n}_t=0$, it cannot
again hit 0 before time $\tau_n$ without first achieving the level
$X^{i,n}_t=2^{-n}$.  To see this, first assume that when $X^{i,n}_t=0$, we
have $Y^{i,n}_t>0$.  The case $Y^{i,n}_t<0$ is similar and will be omitted.
As long as $t<\tau_n$, we have $|Y^{i,n}_t|<2^{n}$ and so $X^{i,n}_t$ has bounded
velocity.  At first, $X^{i,n}_t$ has positive velocity.  If $X^{i,n}_t$ is
ever to reach 0 again, its velocity must change sign, that is, $Y^{i,n}_t$
must reach 0.  But by the lower bound on $|(X^{i,n}_t,Y^{i,n}_t)|_{\ell^\infty}$, if
$Y^{i,n}_t=0$, we have $X^{i,n}_t>2^{-n}$ and since the velocity of
$X^{i,n}_t$ is bounded by $2^{n}$, it follows that $X^{i,n}_t$ takes at
least time $2^{-2n}$ to reach level $2^{-n}$.  Thus, the number of $\sigma^i_k$'s is almost surely bounded.

For simplicity, define $\sigma^i_0=0$.  Also, if $\sigma^i_k$ is the last of
these stopping times, define $\sigma^i_{k+m}=\sigma^i_k$ for $m>0$. \medskip \\
We moreover define
$$
\tilde \sigma^{i}_{k} = \sigma^i_k\wedge \tau_{n}, \quad k=0,1, \cdots, \ i=1,2.
$$
From (\ref{wave-sde-truncated}), it follows that in order to prove Theorem \ref{t1}, it is enough to show the pathwise uniqueness for the solutions of (\ref{wave-sde-truncated}) for any $n\geq1$. We have shown that the sequence of stopping times $\tilde \sigma^i_1<\tilde \sigma^i_2<\cdots$ is a.s. finite for $i=1,2$, therefore the following lemma is the last ingredient in the proof of Theorem \ref{t1}.

 \begin{lemma}
\label{reduced-to}
Assume that $(X^{1,n}_t,Y^{1,n}_t)=(X^{2,n}_t,Y^{2,n}_t)$ for
$t\leq\tilde \sigma^{1}_k$ a.s., and therefore $\tilde \sigma^1_k=\tilde \sigma^2_k$ a.s. Then
$(X^{1,n}_t,Y^{1,n}_t)=(X^{2,n}_t,Y^{2,n}_t)$ for $t\leq\tilde \sigma^{1}_{k+1}$
a.s., and $\tilde \sigma^1_{k+1}=\tilde \sigma^2_{k+1}$ a.s.
\end{lemma}
\begin{proof}
We prove the lemma for $k=0$, that is $\tilde \sigma^{1}_0=0$. The proof for other
values of $k$ is identical.  Furthermore, since (\ref{wave-sde}) is
invariant under the map $(X,Y)\to(-X,-Y)$, we may restrict ourselves to the case
\begin{equation*}
y_0>0.
\end{equation*}
Recall that $|x|^\alpha$ is a Lipschitz continuous function except in a neighborhood of $x=0$. Hence it is enough to prove
the uniqueness of the solutions to (\ref{wave-sde-truncated}) starting at $X^{i,n}_0=0$ up to the first time that either one of $|X^{i,n}_t|$'s hits level $2^{-n}$. Therefore, we can restrict time $t$ to the interval $[0,\eta]$, where $\eta$ is
the first time $t<\tau_n$ at which
\begin{equation*}
|X^{1,n}_t\vee X^{2,n}_t|=2^{-n}.
\end{equation*}
If there is no such time, then $\eta=0$.
Since
$|X^{1,n}|$ and $|X^{2,n}|$ lie in $[0,2^{-n}]$, it follows from the definition of
$\tau_n$ that
\begin{equation*}
Y^{i,n}_t\geq 2^{-n},
\end{equation*}
for $i=1,2$, and therefore $X^{i,n}_t$'s are increasing for $t\in[0,\eta]$. Recall that $Y$ is the velocity of $X$. Since $X^{i,n}_0=0$, we have
\begin{equation} \label{l-bnd}
X^{i,n}_t\geq 2^{-n}t,
\end{equation}
for $i=1,2$ and $t\in[0,\eta]$.  It also follows that
\begin{equation*}
\eta\leq1.
\end{equation*}

Note that
\begin{equation*}
X_t^{i,n}=\int_{0}^{t}\int_{0}^{s}|X^{i,n}_r|^\alpha \mathbf{1}_{[0,\tau_n]}(r)dB_rds
\end{equation*}
and
\begin{equation*}
X_t^{1,n}-X_t^{2,n}=\int_{0}^{t}\int_{0}^{s}\big(|X^{1,n}_r|^\alpha-|X^{2,n}_r|^\alpha\big)
 \mathbf{1}_{[0,\tau_n]}(r)dB_rds.
\end{equation*}
By the Cauchy-Schwarz inequality and Ito's isometry, we get
\begin{align*}
E\left[\left(X^{1,n}_t-X^{2,n}_t\right)^2\right]
&\leq tE\int_{0}^{t}\left(\int_{0}^{s}
 \big(|X^{1,n}_r|^\alpha-|X^{2,n}_r|^\alpha\big)
  \mathbf{1}_{[0,\tau_n]}(r)dB_r\right)^2ds  \\
&= tE\int_{0}^{t}\int_{0}^{s}
 \big(|X^{1,n}_r|^\alpha-|X^{2,n}_r|^\alpha\big)^2
  \mathbf{1}_{[0,\tau_n]}(r)drds    \\
&\leq tE\int_{0}^{t}\int_{0}^{t}
 \big(|X^{1,n}_r|^\alpha-|X^{2,n}_r|^\alpha\big)^2drds  \\
&\leq t^2E\int_{0}^{t}\big(|X^{1,n}_r|^\alpha-|X^{2,n}_r|^\alpha\big)^2dr.
\end{align*}

Now the mean value theorem gives, for $0<a<b$, that for some $c\in(a,b)$
we have
\begin{equation*}
b^\alpha-a^\alpha = \alpha c^{\alpha-1}(b-a) \leq \alpha a^{\alpha-1}(b-a).
\end{equation*}
Thus for $t\in[0,\eta]$, using the lower bound on $X^{i,n}_t$ in (\ref{l-bnd}), we get
\begin{equation*}
\Big||X^{1,n}_r|^\alpha-|X^{2,n}_r|^\alpha\Big|
\leq \alpha(2^{-n}r)^{\alpha-1}\Big||X^{1,n}_r|-|X^{2,n}_r|\Big|.
\end{equation*}

Now let
\begin{equation*}
D_t:=E\Big[\left(|X^{1,n}_r|-|X^{2,n}_r|\right)^2\Big].
\end{equation*}
Since $\eta\leq1$, we get for every $t\in[0,\eta]$,
\begin{equation}
\label{pre-gronwall}
D_t\leq C_n\int_{0}^{t}r^{2\alpha-2}D_rdr
\end{equation}
for some constant $C_n$ depending on $n$.  Since $\alpha>1/2$, we have
$2\alpha-2>-1$ and therefore $r^{2\alpha-2}$ is integrable on $r\in[0,\eta]$.
Since $D_0=0$, Gronwall's lemma implies that $D_t=0$ for all $t\in[0,\eta]$.
This ends the proof of Lemma \ref{reduced-to}, and also the proof of Theorem
\ref{t1}.
\end{proof}

\section{Proof of Theorem \ref{t3}}\label{sec-t3}

Since the solution is starting at $(x_0,y_0)=(0,0)$, we see that
$(X_t,Y_t)\equiv(0,0)$ is a solution to (\ref{wave-sde}).  Our goal is to exhibit
another solution, but this will be a weak solution.  To gain information
about strong uniqueness, we recall the following lemma of Yamada and
Watanabe (see V.17, Theorem 17.1 of Rogers and Williams \cite{rw87}).

\begin{lemma}[Yamada and Watanabe]
\label{yamada-watanabe-lemma} Let $\sigma$ and $b$ be previsible path functionals, and consider the SDE:
\begin{equation}
\label{yw-sde}
dX_{t}=\sigma(t,X_\cdot)dB_{t}+b(t,X_\cdot)dt.
\end{equation}
Then this SDE is exact if and only if the following two conditions hold:
\begin{enumerate}
\item The SDE (\ref{yw-sde}) has a weak solution,
\item The SDE (\ref{yw-sde}) has the pathwise uniqueness property.
\end{enumerate}
Uniqueness in law then holds for (\ref{yw-sde}).
\end{lemma}
Rogers and Williams define \textit{exact} in V.9, Definition 9.4, but it is
not important for our purposes.  Here, $X,b\in\mathbf{R}^n$ and $\sigma\sigma^T$ takes
values in the space of nonnegative definite $n\times n$ matrices.

We already have a weak solution to (\ref{wave-sde}), namely $(X_t,Y_t)\equiv(0,0)$.
So, if we can exhibit a weak solution which is nonzero, then by Lemma
\ref{yamada-watanabe-lemma}, pathwise uniqueness must fail.

Now we construct a nonzero weak solution to (\ref{wave-sde}).
Since \[Y_t=\int_{0}^{t}|X_s|^\alpha dB_s\] is a one-dimensional stochastic integral, it
follows that $Y_t$ is a time-changed Brownian motion.  In particular, if we define
\begin{equation} \label{t-change}
T(t):=\int_{0}^{t}|X_s|^{2\alpha}ds,
\end{equation}
then
\begin{equation*}
\tilde{B}_t:=Y_{T^{-1}(t)}
\end{equation*}
is a standard Brownian motion
as long as
\begin{equation}\label{inverse-t-change}
T^{-1}(t) = \inf\{s\geq 0: T(s)>t\}
\end{equation} is well-defined.
\\\\
We also define
\begin{align}
\label{def-tilde}
\tilde{X}_t&:=X_{T^{-1}(t)}\\
\tilde{Y}_t&:=Y_{T^{-1}(t)}=\tilde{B}_t.  \nonumber
\end{align}
Then, by the chain rule and the inverse function differentiation rule,
\begin{equation*}
d\tilde{X}_t=\tilde{Y}_t|\tilde{X}_t|^{-2\alpha}dt,
\end{equation*}
with the same initial conditions as before.  Thus,
\begin{equation*}
|\tilde{X}_t|^{2\alpha}d\tilde{X}_t=\tilde{Y}_t dt.
\end{equation*}

Let
\begin{equation}
\label{dife-h}
h(x):=\frac{1}{2\alpha+1}|x|^{2\alpha+1}\text{sgn$(x)$}
\end{equation}
and observe that
\begin{equation}
\label{diff-h}
dh(x)=|x|^{2\alpha}dx.
\end{equation}
Since we are assuming that $\alpha>0$, it follows that $dh(0)=0$ and
(\ref{diff-h}) holds for $x=0$.  It is easy to check that (\ref{diff-h}) also holds when $x>0$
and $x<0$.

Let
\begin{equation} \label{V-def}
\tilde{V}_t:=h(\tilde{X}_t).
\end{equation}
Then from (\ref{def-tilde}), we have
\begin{equation} \label{VY-eq}
\begin{aligned}
d\tilde{V}_{t}&=\tilde{Y}_{t}dt \\
d\tilde{Y}_{t}&=d\tilde{B}_{t}
\end{aligned}
\end{equation}
and therefore
\begin{equation} \label{h-X-eq}
h(X_{T^{-1}(t)})=h(\tilde{X}_t)=\tilde{V}_t=\int_{0}^{t}\tilde{B}_sds.
\end{equation}
Note that for any $t>0$, we have
$$
P\Big(\tilde{X}_t \not = 0\Big)=P\Big(\tilde{V}_t \not = 0\Big)=1.
$$
So, in order to prove that $X$ can escape from 0, it is enough to show that $T^{-1}(t)<\infty$ for some $t>0$,
with positive probability.
\medskip \\
Let $t>0$. Then, rewriting $T^{-1}(t)$ using the inverse function derivative,
\begin{equation}  \label{T-int}
\begin{aligned}
T^{-1}(t)&=\int_{0}^{t}\frac{d}{ds}T^{-1}(s)ds  \\
&=\int_{0}^{t}\frac{1}{|X_{T^{-1}(s)}|^{2\alpha}}ds\\
&=\int_{0}^{t}\frac{1}{|\tilde{X}_{s}|^{2\alpha}}ds  \\
&=\int_{0}^{t}\left|h^{-1}\left(\int_{0}^{s}\tilde{B}_rdr\right)\right|
  ^{-2\alpha}ds  \\
&=C\int_{0}^{t}\left|\int_{0}^{s}\tilde{B}_rdr\right|
  ^{-\frac{2\alpha}{2\alpha+1}}ds.
\end{aligned}
\end{equation}
The following lemma, which will be proved at the end of this section, helps us to bound the above integral.
\begin{lemma}\label{lemma-brownian-int}
If $0<\beta<2/3$, then for any $\delta>0$,
\begin{equation*}
I_\beta(\delta)=I(\delta):
 =\int_{0}^{\delta}\left|\int_{0}^{t}B_sds\right|^{-\beta}ds<\infty
\end{equation*}
almost surely.
\end{lemma}
By the assumptions of Theorem \ref{t3}, $0<\alpha<1$. Since this  is equivalent to
\begin{equation*}
0<\frac{2\alpha}{2\alpha+1}<\frac{2}{3},
\end{equation*}
thanks to Lemma \ref{lemma-brownian-int}, the integral in (\ref{T-int}) is finite almost surely.
This finishes the proof of nonuniqueness.
\qed
\medskip \\
\begin{proof}[Proof of Lemma \ref{lemma-brownian-int}]
We check that for all $t>0$ and for $0<\beta<2/3$,
\begin{equation*}
E\left[I(t)\right]<\infty.
\end{equation*}
Let
\begin{equation} \label{J-int}
J_t:=\int_{0}^{t}B_sds.
\end{equation}
Note that $J_{t}$ is a normal random variable with mean 0. Next we compute its
variance.
\begin{equation} \label{var-brownian-int}
\begin{aligned}
\text{Var}(J_t) &= E\left[\left(\int_{0}^{t}B_sds\right)^2\right]  \\
&=\int_{0}^{t}\int_{0}^{t}E\left[B_{r}B_{s}\right]drds  \\
&=2\int_{0}^{t}\int_{0}^{s}E\left[B_rB_s\right]drds  \\
&=2\int_{0}^{t}\int_{0}^{s}rdrds  \\
&=2\int_{0}^{t}\frac{s^2}{2}ds  \\
&=\frac{t^3}{3}.
\end{aligned}
\end{equation}
Now let $Z\sim N(0,1)$ be a standard normal random variable.  From (\ref{var-brownian-int}), it follows that
\begin{equation*}
J_t\stackrel{\mathcal{D}}{=}Ct^{3/2}Z
\end{equation*}
and so
\begin{align*}
E\left[\left|\int_{0}^{t}B_sds\right|^{-\beta}\right]
=Ct^{-3\beta/2}E\Big[|Z|^{-\beta}\Big].
\end{align*}
First, if $\beta<2/3$ then
\begin{equation*}
E\Big[|Z|^{-\beta}\Big]
= C\int_{-\infty}^{\infty}|x|^{-\beta}\exp\left(-\frac{x^2}{2}\right)dx
 <\infty.
\end{equation*}
Secondly,
\begin{align*}
E[I(\delta)]&=\int_{0}^{\delta}
 E\left[\left|\int_{0}^{t}B_sds\right|^{-\beta}\right]dt \\
&=C\int_{0}^{\delta}t^{-3\beta/2}dt \\
&<\infty
\end{align*}
provided $3\beta/2<1$, which is equivalent to $\beta<2/3$.
\end{proof}

\section{Proof of Theorem \ref{t2}}\label{sec-t2}

Fix the initial point $(x_0,y_0)\ne(0,0)$, and let
\begin{equation*}
Z_t:=\left(B_t,\int_{0}^{t}B_sds\right)=(B_t,J_t).
\end{equation*}
We need to study the joint distribution of the components $B_t$ and
$\int_{0}^{t}B_sds$, which are jointly centered Gaussian.  Using
(\ref{var-brownian-int}) and by a simple calculation, we find that the
covariance matrix of $(B_t,J_t)$ is
\begin{equation*}
M_t=\left(
\begin{matrix}
t & t^2/2 \\
t^2/2 & t^3/3
\end{matrix}
\right)
\end{equation*}
and
\begin{equation*}
\det(M_t)=\frac{t^4}{12}.
\end{equation*}
Since $(B_t,J_t)$ is jointly Gaussian, its joint probability density has the
following bound.
\begin{equation} \label{densityJB}
f_{B_t,J_t}(x,y)
=\frac{\exp\left[-(x,y)M_t^{-1}(x,y)^T\right]}{\sqrt{(2\pi)^2t^4/12}}
\leq\frac{1}{\sqrt{(2\pi)^2t^4/12}}
\leq t^{-2}.
\end{equation}

We define the following events
\begin{align*}
A&=\{Z_t=(0,0)\text{ for some $t>0$}\}  \\
A_N&=\{Z_t=(0,0)\text{ for some $t\in[1/N,N]$}\}
\end{align*}
for natural numbers $N$.
We wish to prove that $P(A)=0$, and it is enough to prove that $P(A_N)=0$ for all $N$.  From now on, let $N$ be fixed.

Fix $0<\delta <1$ and let $k,m,n$ be natural numbers.  We
define a few more events:
\begin{align*}
E_{1,n,N}&=\left\{\sup_{1/N<t<N}|B_t|\leq n\right\},  \\
E_{2,k,n}^{c}&=\{|B_{k2^{-2n}}|\leq2^{-n(1-\delta)},|J_{k2^{-2n}}|\leq2^{-2n(1-\delta)}\}, \\
E_{3,n,N}&=\bigcap_{k:\;k2^{-2n}\in[1/N,N]}E_{2,k,n},  \\
E_{4,k,n}&=\left\{\sup_{t\in[k2^{-2n},(k+1)2^{-2n}]}
   |B_t-B_{k2^{-2n}}|<2^{-n(1-\delta)}\right\},  \\
E_{5,n,N}&=\bigcap_{k:\;k2^{-2n}\in[1/N,N]}E_{4,k,n},  \\
E_{6,k,n}&=\left\{\sup_{t\in[k2^{-2n},(k+1)2^{-2n}]}
   |J_t-J_{k2^{-2n}}|<2^{-2n(1-\delta)}\right\},  \\
E_{7,n,N}&=\bigcap_{k:\;k2^{-2n}\in[1/N,N]}E_{6,k,n}.
\end{align*}
As $k$ varies, $k2^{-2n}$ is a grid of points which gets
denser as $n$ increases.

Next, note that
\begin{equation*}
\lim_{n\to\infty}P(E_{1,n,N}^c)=0.
\end{equation*}
From (\ref{densityJB}) we have for all $k2^{-2n}\geq1/N$
\begin{equation*}
P(E_{2,k,n}^c)\leq 4\cdot2^{-3n(1-\delta)}N^2,
\end{equation*}
and therefore
\begin{equation*}
P(E_{3,n,N}^c)\leq 4N2^{2n}\cdot 2^{-3n(1-\delta)}N^2
=4N^32^{-n+3\delta}.
\end{equation*}

To deal with $E_{5,n,N}$, recall that L\'evy's modulus of continuity for
Brownian motion (see M\"orters and Peres \cite{mper10}, Theorem 1.14) states
that for $T>0$ fixed, we have
\begin{equation} \label{mod-con}
\lim_{n\to\infty}\sup_{0<h\leq2^{-2n}}\sup_{0\leq t\leq T-h}
\frac{|B_{t+h}-B_t|}{\sqrt{2h\log\log(h)}}
=1, \quad\text{ a.s.,}
\end{equation}
and therefore
\begin{equation*}
\lim_{n\to\infty}P(E_{5,n,N}^c)=0.
\end{equation*}

Now we deal with $J_t$.  Note that on $E_{1,n,N}$, the velocity of $J_t$ is
bounded by $n$ in absolute value.  It follows that on $E_{1,n,N}$, all of the
$E_{6,k,n}$'s occur and so on $E_{1,n,N}$, $E_{7,n,N}$ also occurs.

Observe that on $E_{3,n,N}\cap E_{5,n,N}\cap E_{7,n,N}$ we have
$(B_t,J_t)\ne0$ for $1/N<t<N$.  Also, by the above we have
\begin{equation*}
\lim_{n\to\infty}P(E_{1,n,N}\cap E_{3,n,N}\cap E_{5,n,N}\cap E_{7,n,N})
=1.
\end{equation*}
It follows that
\begin{equation*}
P\big((B_t,J_t)\ne0\text{ for $1/N<t<N$}\big)=1.
\end{equation*}
Since $N$ was arbitrary, this finishes the proof of Theorem \ref{t2}.
\qed
\section{Proof of Theorem \ref{t4} } \label{sec-t4}
The proof of Theorem \ref{t4} contains two main ingredients. Recall that in Section \ref{sec-t3}, we showed that a solution of system (\ref{wave-sde}) with $0<\alpha<1$ and $(x_0,y_0)=(0,0)$ can be represented as a time change of $(B_t,J_{t})$, where $J_{t}$ was defined in (\ref{J-int}).
%:
%\begin{equation*}
%J_t=\int_0^tB_s\,ds.
%\end{equation*}
In Proposition \ref{prop-trans}, we will prove that $(B_t,J_{t})$ is transient. In Lemma~\ref{lemmaforthm4}, we will prove that when $\alpha>1$ and $(x_0,y_0) \neq (0,0)$, the inverse time  change $T^{-1}(t)$  in \eqref{inverse-t-change} satisfies $P\left(\sup_{t>0} T^{-1}(t)<+\infty \right)=1$.  In other words, the time change $T^{-1}(t)$ changes infinite time to finite time almost surely, and this will complete the proof of Theorem \ref{t4}.
\begin{proposition} \label{prop-trans}
Let $\{B_{t}\}_{t\geq 0}$ be a one-dimensional Brownian motion starting from 0. Then the spatial process $\{(B_t,J_{t})\}_{t\geq 0}$ is transient.
\end{proposition}
\begin{proof}
Let $0<\delta_1<\delta_2<\delta_3<1/2$ and $0<\delta_4<1/2-\delta_3$.  We define the following events
\begin{align*}
A^c_{1,n}&=\left\{\left|B_{n^2}\right| \leq n^{1-\delta_3},\;
 |J_{n^2}|\leq n^{2+\delta_2}\right\},  \\
A_{2,N}&=\bigcap_{n=N}^\infty A_{1,n} ,  \\
A_{3,n}&=\bigg\{\sup_{n^2 \leq t \leq (n+1)^2}\left|B_t-B_{n^2}\right|<n^{1/2+\delta_4}\bigg\} ,\\
A_{4,N}&=\bigcap_{n=N}^\infty A_{3,n},  \\
A_{5,n}&=\bigg\{\sup_{n^2\leq t\leq(n+1)^2}|J_t-J_{n^2}|< n^{2+\delta_1}\bigg\},  \\
A_{6,N}&=\bigcap_{n=N}^\infty A_{5,n}.
\end{align*}
Note that $(B_t,J_t)$ is transient on the set $A_{2,N}\cap A_{4,N}\cap A_{6,N}$. We now show that the probability of this set tends to 1 as $N\rightarrow \infty $.

Using inequality (\ref{densityJB}), we get
\begin{equation*}
\label{eq-e-1}
P(A^c_{1,n})\leq C(n^2)^{-2}n^{3-\delta_3+\delta_2}
=Cn^{-1-\delta_3+\delta_2}.
\end{equation*}
It follows from a comparison principle that
\begin{equation} \label{A-2}
P(A^c_{2,N})\le \sum_{n \ge N} P(A^c_{1,n}) \leq CN^{-\delta_3+\delta_{2}} \rightarrow 0,
\end{equation}
as $N \rightarrow \infty$, since $\delta_2<\delta_3$.

A bound of the probability of the event $A^{c}_{3,n}$  can be computed by time change and reflection principle:
\begin{align*}
 P(A^{c}_{3,n})&=P\left(\sup_{n^2 \leq t \leq (n+1)^2}\left|B_t-B_{n^2}\right|\ge n^{1/2+\delta_4}\right)\\
 &=P\left(\sup_{0 \leq t \leq 2n+1}\left|B_t\right|\ge n^{1/2+\delta_4}\right)\\
 &=P\left(\sup_{0 \leq t \leq 1}\left|B_t\right|\ge \frac{n^{1/2+\delta_4}}{\sqrt{2n+1}}\right)\le P\left(\sup_{0 \leq t \leq 1}\left|B_t\right|\ge \frac{1}{\sqrt{3}}n^{\delta_4}\right)\\
 &\le 4P\left(B_1\ge \frac{1}{\sqrt{3}}n^{\delta_4}\right)\le C \exp\left\{-\frac{2}{3}n^{2\delta_4}\right\}.
\end{align*}
It follows that
\begin{equation} \label{A-4}
P(A^{c}_{4,N}) \le \sum_{n \ge N} P(A^c_{3,n}) \rightarrow 0
\end{equation}
as $N \rightarrow \infty$.

By the law of iterated logarithm for Brownian motion (see e.g. Theorem 5.1 in \cite{mper10}), there exists $N_*>0$ such that for all $n \ge N_*$,
\begin{equation*}
\sup_{n^2\leq t\leq(n+1)^2}|J_t-J_{n^2}| \le (2n+1) \sup_{n^2\leq t\leq(n+1)^2}|B_{t}|\le n^{2+\delta_1}
\end{equation*}
almost surely. It follows that
\begin{equation}\label{A-6}
\lim_{N\rightarrow \infty} P(A_{6,N})= 1.
\end{equation}
From (\ref{A-2})--(\ref{A-6}) we get
\begin{equation*}
\lim_{N\rightarrow \infty}P(A_{2,N}\cap A_{4,N}\cap A_{6,N}) =1,
\end{equation*}
and the conclusion that $(B_{t},J_{t})$ is transient follows.
\end{proof}
\begin{remark}  \label{remark-J-B}
From the proof of Proposition \ref{prop-trans}, we can get a lower bound on the growth rate of $(B_{t},J_{t})$.
Since the time intervals $[n^2,(n+1)^2]$ are of lengths $2n+1$, the fluctuations
of $B_t$ over such intervals are of order $n^{1/2+\delta_4}<<n^{1-\delta_3}$
for large values of $n$.  This assertion holds because $0<\delta_3<1/2$ and $0<\delta_4<1/2-\delta_3$.
So the fluctuations won't bring $B_t$ to 0, if it is not already close to 0.

As for $J_t$, on the time intervals $[n^2,(n+1)^2]$, the fluctuations of $J_t$ are bounded by
$n^{2+\delta_1}$.  This is of smaller order than
$n^{2+\delta_2}$  since $\delta_1<\delta_2$.

Therefore, for large values of $t$, one of the two inequalities
\begin{align*}
|B_t|&\geq t^{1/2-\delta_3/2}\\
|J_t|&\geq t^{1+\delta_2/2}
\end{align*}
always
holds a.s., where $0<\delta_{2} <\delta_{3}<1/2$.

Note that both $B_t$ and $J_t$ are recurrent processes which return to 0 infinitely often. However, if we consider the collection of the processes $(B_t,J_t)$, if one process takes a small value, the other will take a large value, due to the correlation between them we will eventually have $|(B_t,J_t)|_{\ell^\infty}\rightarrow \infty$ as $t \rightarrow \infty$.
\end{remark}

\paragraph{\textbf{Proof of Theorem \ref{t4}}} Suppose that $\alpha>1$ and the solution $(X_t,Y_t)$ of  \eqref{wave-sde} started from $(x_0,y_0) \neq (0,0)$. Recall that with the definitions for $T(t)$ and $h(x)$ in \eqref{t-change}  and \eqref{dife-h},
 the time-changed process  $(\tilde V_{t},\tilde Y_{t})=(h(X_{T^{-1}(t)}),Y_{T^{-1}(t)})$ defined in \eqref{VY-eq} satisfies
 \begin{equation}  \label{v-transform}
\begin{aligned}
\tilde V_{t}&=h(x_{0})+y_0t+ \int_{0}^{t}\tilde B_sds\\
\tilde Y_{t}&=y_0+\tilde B_{t},
\end{aligned}
\end{equation}
where $\tilde B_{t}$ is a standard one-dimensional Brownian motion.
\medskip \\
Thanks to Proposition \ref{prop-trans}, it is true that $|(\tilde{V}_t,\tilde{Y}_t)|_{\ell^\infty} \rightarrow \infty$ as $t \rightarrow \infty$ almost surely. If we can show that
\begin{equation} \label{fin-T-inv}
P\big(\lim_{t\rightarrow \infty}T^{-1}(t) <\infty \big)=1,
\end{equation}
then blowup in finite time for $( X_{t}, Y_{t})$ will follow. For this purpose, we state Lemma~\ref{lemmaforthm4}.

\begin{lemma}
\label{lemmaforthm4} Suppose $(x_0,y_0) \neq (0,0)$. If  $2/3<\beta<1$, then
$\int_0^\infty |h(x_0)+y_0t+J_t|^{-\beta}dt<\infty$ almost surely.
\end{lemma}

We will prove the Lemma shortly. If we assume for now that Lemma 4 is true, then from (\ref{T-int}) and (\ref{v-transform}) we can derive that
\begin{align*}
\lim_{t\rightarrow \infty}T^{-1}(t) &=\int_{0}^{\infty}\frac{1}{|X_{T^{-1}(t)}|^{2\alpha}}dt\\
&=\int_0^\infty \left|h(x_0)+y_0t+\int_0^t\tilde{B}_sds\right|^{-\frac{2\alpha}{2\alpha+1}}\,dt.
\end{align*}
By applying Lemma~\ref{lemmaforthm4} for $\beta=\frac{2\alpha}{2\alpha+1}$, we can conclude that \eqref{fin-T-inv} is satisfied. Recall that $\alpha>1$, so that $2/3<\beta<1$, which satisfies the condition for Lemma~\ref{lemmaforthm4}.
\qed
\medskip \\

For the proof of Lemma~\ref{lemmaforthm4}, we first require an alternative 
representation of the expectation $E|X|^{-\beta}$,  where 
$X\sim\mathcal{N}(m,\sigma^2)$ and $0< \beta <1$. We write the integral 
representation of  a confluent hypergeometric function in 
Lemma~\ref{lem:gauss-moment}. Even though this expression is already 
well-known, the authors couldn't find a good reference for it (see 
\cite{win12} and Ch 13 of \cite{as65}). So we give a direct proof of the 
lemma as well.

\begin{lemma} \label{lem:gauss-moment}
  Let $Z$ be a standard $\mathcal{N}(0,1)$ random variable and let $m\in \mathbb{R}$ and $\sigma^2>0$. Then for any $0<\beta<1$,
\begin{equation*}
    E|m + \sigma Z|^{-\beta} = \frac{(2\sigma^2)^{-\beta/2}}{\Gamma(\beta/2)} \int_0^1 e^{-\frac{m^2u}{2\sigma^2}} u^{\beta/2-1}(1-u)^{-\beta/2-1/2}du.
\end{equation*}
\end{lemma}

\begin{proof}

First, we prove that if $\xi$ is a nonnegative random variable, then for any $\alpha$ such that the integral converges
\begin{equation}\label{lem:alpha-power2}
    E(\xi^{-\alpha}) = \frac{1}{\Gamma(\alpha)}\int_0^\infty E(e^{-\lambda \xi}) \lambda^{\alpha-1}d\lambda.
\end{equation}

By switching the order of integration and by a change of variables $t = \lambda \xi$ we get
\begin{equation*}
    \int_0^\infty E(e^{-\lambda \xi}) \lambda^{\alpha-1}d\lambda =E \int_0^\infty e^{-t} t^{\alpha-1} \xi^{-\alpha} dt = \Gamma(\alpha)E(\xi^{-\alpha}).
\end{equation*}

Second, we prove that if $Z\sim \mathcal{N}(0,1)$, then the Laplace transform of $|m + \sigma Z|^2$ is for any $\lambda>0$,
\begin{equation} \label{eq:Gauss-Laplace-tran}
  E e^{-\lambda |m + \sigma Z|^2} = \frac{e^{-\frac{\lambda m^2}{1 + 2 \lambda \sigma^2}}}{\sqrt{1 + 2\lambda \sigma^2}}.
\end{equation}
  \begin{align*}
    E e^{-\lambda |m + \sigma Z|^2} &= \frac{1}{\sqrt{2\pi}} \int_{-\infty}^\infty e^{-\lambda m^2 - 2m\lambda\sigma x -\lambda \sigma^2 x^2 - \frac{1}{2}x^2}dx\\
    &= \frac{e^{-\lambda m^2}e^{\frac{2\lambda^2m^2\sigma^2}{1 + 2\lambda \sigma^2}}}{\sqrt{2\pi}} \int_{-\infty}^\infty e^{-\frac{1}{2}(1 + 2\lambda \sigma^2)\left(x^2 +\frac{4\lambda m \sigma x}{1 + 2\lambda \sigma^2} + \frac{4\lambda^2m^2\sigma^2}{(1 +2\lambda \sigma^2)^2} \right)}dx\\
    &=\frac{e^{-\frac{\lambda m^2}{1 + 2 \lambda \sigma^2}}}{\sqrt{2\pi}} \int_{-\infty}^\infty e^{-\frac{1}{2}(1 + 2 \lambda \sigma^2) \left(x + \frac{2\lambda m \sigma}{1 + 2 \lambda \sigma^2} \right)^2}dx\\
    &=\frac{e^{-\frac{\lambda m^2}{1 + 2\lambda \sigma^2}}}{\sqrt{1 + 2 \lambda \sigma^2}}.
  \end{align*}

Now, we are ready to prove the main result. By \eqref{lem:alpha-power2} and  \eqref{eq:Gauss-Laplace-tran},
\begin{align*}
E|m + \sigma Z|^{-\beta} &= E \left(|m + \sigma Z|^2 \right)^{-\beta/2}\\
& = \frac{1}{\Gamma(\beta/2)} \int_0^\infty E\left(e^{-\lambda|m +\sigma Z|^2} \right) \lambda^{\beta/2-1}d\lambda\\
&=\frac{1}{\Gamma(\beta/2)}\int_0^\infty \frac{e^{-\frac{\lambda m^2}{1+2\lambda\sigma^2}}}{\sqrt{1 + 2\lambda \sigma^2}}\lambda^{\beta/2 -1} d\lambda.
\end{align*}
We make the following change of variables 
\[u
=\frac{2\lambda \sigma^2}{1+2\lambda \sigma^2}.
\]
Notice that
\[
\lambda = \frac{u}{(2\sigma^2)(1-u)}
\]
and
\[du = \frac{2\sigma^2}{(1 + 2\lambda \sigma^2)^2} d\lambda.\]
Under this change of variables we have
\begin{align*}
\frac{\lambda^{\beta/2 -1}d\lambda}{\sqrt{1+2\lambda \sigma^2}}& = \frac{(1 + 2\lambda \sigma^2)^{3/2} \lambda^{\beta/2 +1/2}}{2\sigma^2 \lambda^{3/2}}du\\
&=(2\sigma^2)^{1/2}u^{-3/2} \left(\frac{u}{2\sigma^2(1-u)} \right)^{\beta/2+1/2}du\\
& = (2\sigma^2)^{-\beta/2}u^{\beta/2 -1}(1-u)^{-\beta/2-1/2}du.\end{align*}
Therefore, Lemma~\ref{lem:gauss-moment} follows.
\end{proof}

We are now ready to prove Lemma~\ref{lemmaforthm4}.

\begin{proof}[Proof of Lemma \ref{lemmaforthm4}]
We show that
\begin{equation}\label{whattoprove}
E\int_0^\infty |h(x_0)+y_0t+ J_t|^{-\beta}dt=\int_0^\infty E|h(x_0)+y_0t+J_t|^{-\beta}dt<\infty
\end{equation}
for  $2/3<\beta<1$.

Note that from equation \eqref{var-brownian-int}, $h(x_0)+y_0t+J_t$ is a normal random variable with mean $h(x_0)+y_0t$ and variance $t^3/3$. By Lemma \ref{lem:gauss-moment}, for $t>0$, we may write $E|h(x_0)+y_0t+J_t|^{-\beta}$ as the integral representation of  a confluent hypergeometric function.
\begin{align*}
E|h(x_0)+y_0t+J_t|^{-\beta}=&C_1 t^{-\frac{3}{2}\beta}\int_0^1 \exp\{-C_2u(h(x_0)+y_0t)^2t^{-3}\}\\
&\times u^{\frac{\beta}{2}-1}(1-u)^{-\frac{\beta}{2}-\frac{1}{2}} du\\
=& C_1 \int_0^1 t^{-\frac{3}{2}\beta}\exp\{-C_2uf(t)\}g(u) \,du.
\end{align*}
Here, $C_1$ and $C_2$ are positive constants depending on $\beta$,
\[f(t)=(h(x_0)+y_0t)^2t^{-3},\] and
\[g(u)=u^{\frac{\beta}{2}-1}(1-u)^{-\frac{\beta}{2}-\frac{1}{2}}.\]

First, we consider the term $\exp\{-C_2uf(t)\}$. Note that since $(x_{0},y_{0}) \not =(0,0)$, we have
\begin{equation*}
\lim_{t\rightarrow 0}tf(t) >0,  \quad \lim_{t\rightarrow \infty } t^{3}f(t) >0.
\end{equation*}
%the function $f(t)$ is of order between $t^{-1}$ and $t^{-3}$ as $t$ goes to $0$ or $\infty$. In particular, if $h(x_0)=0$, then $f(t)$ is of order $t^{-1}$, if $y_0=0$, then $f(t)$ is of oder $t^{-3}$, and for the rest of the cases, $f(t)$ is of order $t^{-3}$ as $t \rightarrow 0$ and of order $t^{-1}$ as $t \rightarrow \infty$.
So, it is possible to find  positive constants $C_3,\cdots,C_6$ such that
\[\exp\{-C_2uf(t)\}\le C_3\exp\{-C_4ut^{-1}\}+C_5\exp\{-C_6ut^{-3}\}\]
for all $t>0$. So, to prove \eqref{whattoprove}, we only need to show the convergence of the integrals of the terms on the right, which are the cases of $k(t)=t^{-1}$ and $k(t)=t^{-3}$.

Let's first consider the first term, so $k(t)=t^{-1}$. Without loss of generality, we may assume that $C_3=C_4=1$. Then, we show that
\begin{align}
\notag\int_0^\infty \int_0^1 t^{-\frac{3}{2}\beta}\exp\{-u/t\}g(u) \,du&\,dt=\\
\label{whattoprove2}\int_0^1 &\left(\int_0^\infty t^{-\frac{3}{2}\beta}\exp\{-u/t\}\,dt\right)g(u)\, du
\end{align}
is finite.
	
By a change of variables $v=u/t$, we get for the integral with respect to $t$
\begin{align*}
\int_0^\infty t^{-\frac{3}{2}\beta}\exp\{-u/t\}\,dt&=\int_0^\infty \frac{u^{1-\frac{3}{2}\beta}}{v^{2-\frac{3}{2}\beta}}\exp\{-v\}\,dv\\
&=u^{1-\frac{3}{2}\beta}\int_0^\infty \frac{1}{v^{2-\frac{3}{2}\beta}}\exp\{-v\}\,dv\\
&=Cu^{1-\frac{3}{2}\beta}
\end{align*}
for some constant $C>0$. Note that the integral
\[\int_0^\infty \frac{1}{v^{2-\frac{3}{2}\beta}}\exp\{-v\}\,dv\]
is finite because $2-3\beta/2<1$, which is equivalent to $\beta>2/3$. Now, \eqref{whattoprove2} becomes
\begin{align*}
C\int_0^1u^{1-\frac{3}{2}\beta}g(u)\, du=C\int_0^1u^{-\beta}(1-u)^{-\frac{\beta}{2}-\frac{1}{2}}\,du.
\end{align*}
This integral is finite if and only if $-\beta>-1$ and $-\frac{\beta}{2}-\frac{1}{2}>-1$, which are equivalent to $\beta<1$.

We can use an analogous method for solving the problem in the case $k(t)=t^{-3}$. Then, we get the conclusion that
\begin{equation*}
\int_0^1 \left(\int_0^\infty t^{-\frac{3}{2}\beta}\exp\{-u/t^3\}\,dt\right)g(u)\, du<\infty
\end{equation*}
if and only if  $\frac{4}{3}-\frac{1}{2}\beta<1$, and $-\frac{\beta}{2}-\frac{1}{2}>-1$, which are equivalent to $2/3<\beta<1$.

One final remark is that the interchanges of the orders of the integrals in the proof are justified by the Fubini's theorem after proving finiteness of the integrals.
\end{proof}

%\bibliography{bibtex}
%\bibliographystyle{amsalpha}

\def\cprime{$'$} \def\cprime{$'$} \def\cprime{$'$}
\providecommand{\bysame}{\leavevmode\hbox to3em{\hrulefill}\thinspace}
\providecommand{\MR}{\relax\ifhmode\unskip\space\fi MR }
\providecommand{\MRhref}[2]{%
  \href{http://www.ams.org/mathscinet-getitem?mr=#1}{#2}
}
\providecommand{\href}[2]{#2}

\end{document}